\documentclass[11pt,reqno]{amsart}

\usepackage{amsthm, mathrsfs, amsmath, amstext, amsxtra, amsfonts, dsfont, amssymb, bm}
\usepackage{xcolor}
\usepackage{lmodern}
\usepackage{relsize}
\usepackage[T1]{fontenc}
\usepackage[colorlinks, linkcolor=red, citecolor=blue, urlcolor=blue, hypertexnames=false]{hyperref}
\usepackage{tikz,tkz-tab}
\usepackage{soul}

\usepackage{geometry}
\newcommand{\beq}{\begin{eqnarray}}
\newcommand{\eeq}{\end{eqnarray}}
\newcommand{\bq}{\begin{equation}}
\newcommand{\eq}{\end{equation}}
\newcommand{\beqn}{\begin{eqnarray*}}
\newcommand{\eeqn}{\end{eqnarray*}}
\newcommand{\dint}{\displaystyle\int}

\newcommand{\R}{\ensuremath{\mathbb{R}}}

\newcommand{\vertiii}[1]{{\vert\kern-0.25ex\vert\kern-0.25ex\vert #1
    \vert\kern-0.25ex\vert\kern-0.25ex\vert}}

\newcommand{\ignore}[1]{}

\newtheorem{definition}{Definition}[section]
\newtheorem{proposition}{Proposition}[section]
\newtheorem{theorem}{Theorem}[section]

\newtheorem{lemma}{Lemma}[section]

\title[{\tt IBVP} for nonlinear hyperbolic equations]{Existence of solutions for a class of {IBVP} for nonlinear hyperbolic equations}

\author[S. G. Georgiev, M. Majdoub]{Svetlin Georgiev Georgiev \& Mohamed Majdoub}
\address{Department of Mathematics, College of Science, Imam Abdulrahman Bin Faisal University, P. O. Box 1982, Dammam, Saudi Arabia}
\address{Basic and Applied Scientific Research Center, Imam Abdulrahman Bin Faisal University, P.O. Box 1982, 31441, Dammam, Saudi Arabia}
\email{\sl svetlingeorgiev1@gmail.com}

\email{\sl mmajdoub@iau.edu.sa}

\begin{document}

\begin{abstract} We study  a class of initial  boundary value problems  of hyperbolic type. A new topological approach is applied to prove the existence of non-negative classical solutions.  The arguments are based upon   a recent  theoretical result.
\end{abstract}


\subjclass[2010]{47H10,  58J20,  35L15}

\keywords{Hyperbolic Equations, positive solution, fixed point, cone, sum of operators.}


\date{\today}

\maketitle

\section{Introduction}
This paper concerns global existence of classical solutions of  one-dimensional  nonlinear wave equations with initial and  mixed boundary conditions.  More precisely, we investigate the following IBVP
\begin{equation}\label{1}
\left\{\begin{array}{ccll}
u_{tt}-u_{xx}&=& f(t, x, u),\quad t\geq 0,\quad x\in [0, L],\\
u(0, x)&=& u_0(x),\quad x\in [0, L],\\
u_t(0, x)&=& u_1(x),\quad x\in [0, L],\\
u(t, 0)&=& u_x(t, L)=0,\quad t\geq 0,
\end{array}
\right.
\end{equation}
where $L>0$, $f : [0,\infty)\times [0,L]\times \R\to\R$ is continuous and $u_0\in \mathcal{C}^2([0, L])$, $u_1\in \mathcal{C}^1([0, L])$ are the initial data.

Mixed boundary value problems arise in several areas of applied mathematics and physics,
such as gas dynamics, nuclear physics, chemical reaction, studies of atomic structures, and atomic
calculation. Therefore, mixed problems have attracted much interest and have been studied by many
authors. See \cite{Arg, BCN, BMMW} and references cited therein. In our case, the equation \eqref{1} describes the interaction of solitary waves in elastic rods, the dynamics of one-dimensional internal gravity waves in an incompressible stratified fluid.

In \cite{godin} the author investigate the following mixed problem
\begin{equation}
\label{2.1}
\left\{\begin{array}{ccll}
u_{tt}-u_{xx}&=&f(u),\quad x>0, \quad t>0,\\
u_x+\gamma u_t&=&0,\quad x=0,\quad t>0,\\
u=\psi_0,\quad u_t&=&\psi_1,\quad x>0,\quad t=0,
\end{array}
\right.
\end{equation}
where $|\gamma|\geq 1$, under the compatibility conditions
\begin{equation*}
\begin{array}{l}
\psi_0^{\prime}(0)+\gamma \psi_1(0)=0,\\
\psi_1^{\prime}(0)+\gamma( \psi_0^{\prime\prime}(0)+F(\psi_0(0)))=0,\\
\psi_0^{\prime\prime\prime}(0)+\gamma \psi_1^{\prime\prime}(0)=0.
\end{array}
\end{equation*}
If $f\in \mathcal{C}^1(\mathbb{R})$, $\gamma\ne 1$, $\psi_j\in \mathcal{C}^{2-j}([0, \infty))$, $j=0, 1$, it is proved that there exists an open neighborhood $U$ of $\{0\}\times [0, \infty)$ such that  \eqref{2.1} has exactly one solution $u\in \mathcal{C}^2(\overline{U})$. Moreover, if $f\in \mathcal{C}^2(\mathbb{R})$, $\psi_j\in \mathcal{C}^{3-j}([0, \infty))$ for $j=0, 1$, then there exists an open neighborhood $U$ of $\{0\}\times [0, \infty)$ such that \eqref{2.1} has exactly one solution $u\in \mathcal{C}^3(\overline{U})$. The method use In \cite{godin} is mainly based on  conservation laws.

The following mixed problem  is investigated in \cite{kharibeg}
\begin{equation}
\label{2.2}
\left\{\begin{array}{ccll}
u_{tt}-u_{xx}+g(u)&=& f(x, t),\quad (x, t)\in (0, L)\times (0, T),\\
u(x, 0)=\phi(x),\quad u_t(x, 0)&=& \psi(x),\quad x\in [0, L],\\ \\
u_x(0, t)= F(u(0, t))+\alpha(t),\quad u_x(L, t)&=& \beta(t) u(L, t)+\gamma(t),\quad t\in [0, T],
\end{array}
\right.
\end{equation}
where $g$, $f$, $\phi$, $\psi$, $\alpha$, $\beta$, $\gamma$ and $F$  are given functions. For $f\in \mathcal{C}([0, L]\times [0, T])$, $g\in \mathcal{C}(\mathbb{R})$, $F\in \mathcal{C}^1(\mathbb{R})$, $\phi\in \mathcal{C}^2([0, L])$, $\psi\in \mathcal{C}^1([0, L])$, $\alpha, \beta, \gamma\in \mathcal{C}^1([0, T])$, necessary conditions for solvability of the problem \eqref{2.2} in the class $\mathcal{C}^2([0, L]\times [0, T])$ are given in \cite{kharibeg}. To prove the main results in \cite{kharibeg}, the authors reduce \eqref{2.2} to an equivalent system of Volterra type in the class of continuous functions.

Note also that \eqref{2.2} was studied in \cite{boc, now, vit} in the energy space using Fourier method.\\

In the present paper we propose a new approach based on the fixed point theory on cones for the sum of two operators. Before stating our main result we precise the assumptions made on the nonlinearity and the initial data. We suppose that $f$ is continuous and satisfies

\begin{equation}
\label{H1}
0\leq f(t, x, u)\leq \sum_{j=1}^l c_j(t, x) |u|^{p_j},\quad (t, x, u)\in [0, \infty)\times [0, L]\times \mathbb{R},
\end{equation}
where $p_j>0$, $c_j\in \mathcal{C}([0, \infty)\times [0, L])$, $j\in \{1, \ldots, l\}$, $l\in \mathbb{N}$. For the initial data $u_0, u_1$ we make the following assumption.
\begin{equation}
\label{H2}
\left\{\begin{array}{ll}
u_0\in \mathcal{C}^2([0, L]),\; u_1\in \mathcal{C}^1([0, L]),& u_0(0)=u_{0x}(L)=0,\;  u_1(0)=u_{1x}(L)=0,\\
0\leq u_0, u_1<r\;\;\mbox{on}\;\;[0, L],\;\; u_0>0&\;\;\mbox{on}\;\; [\frac{L}{3}, \frac{L}{2}],
\end{array}
\right.
\end{equation}
where $r\in (0, 1)$.

Our main result reads as follows.
\begin{theorem}
\label{main}
Suppose that assumptions \eqref{H1}-\eqref{H2} are fulfilled.  Then the IBVP \eqref{1} has at least one non-negative solution $u\in \mathcal{C}^2([0, \infty)\times [0, L])$.
\end{theorem}

The set up of the paper is as follows. In the next Section we give some useful tools and preliminary results. In Section 3, we prove our main result. In the last section, Section 4, we give an example.

\section{Background and Preliminary Results}
Let $X$ be a real Banach space.
\begin{definition}
A mapping $K: X\to X$ is said to be completely continuous if it is continuous and maps bounded sets into relatively compact sets.
\end{definition}
The concept for $k$-set contraction  is related to that of the Kuratowski measure of noncompactness which we recall for completeness.
\begin{definition}
Let $\Omega_X$ be the class of all bounded sets of $X$. The Kuratowski measure of noncompactness $\alpha: \Omega_X\to [0, \infty)$ is defined by
\begin{equation*}
\alpha(Y)= \inf\left\{ \delta>0: Y= \bigcup_{j=1}^m Y_j\quad \text{and}\quad \text{diam}(Y_j)\leq \delta,\quad j\in \{1, \ldots, m\}\right\},
\end{equation*}
where $\text{diam}(Y_j)=\sup\{\Vert x-y\Vert_X: x, y\in Y_j\}$ is the diameter of $Y_j$, $j\in \{1, \ldots, m\}$.
\end{definition}
For the main properties of measure of noncompactness we refer the reader to \cite{deim}.
\begin{definition}
A mapping $K: X\to X$ is said to be $k$-set contraction if there exists a constant $k\geq 0$ such that
\begin{equation*}
\alpha(K(Y))\leq k\alpha(Y)
\end{equation*}
for any bounded set $Y\subset X$.
\end{definition}
Obviously, if $K: X\to X$ is a completely continuous mapping, then $K$ is $0$-set contraction(see \cite{drabek}).
\begin{definition} Let $X$ and $Y$ be real Banach spaces.
A mapping $K: X\to Y$ is said to be expansive if there exists a constant $h>1$ such that
\begin{equation*}
\Vert Kx-Ky\Vert_Y \geq h\Vert x-y\Vert_X
\end{equation*}
for any $x, y\in X$.
\end{definition}
\begin{definition}
A closed, convex set $\mathcal{P}$ in $X$ is said to be cone if
\begin{enumerate}
\item $\alpha x\in \mathcal{P}$ for any $\alpha\geq 0$ and for any $x\in \mathcal{P}$,
\item $x, -x\in \mathcal{P}$ implies $x=0$.
\end{enumerate}
\end{definition}
The following Proposition will be used to prove our main result. We refer the reader to \cite{DjebaMeb} for more detail.

\begin{proposition}\label{prop2c}  Suppose that $\mathcal{P}$ is a cone in $X$. Let $\Omega$ be a subset of $\mathcal{P}$ and  $U$ be a bounded open subset of $\mathcal{P}$ with $0\in U.$
Assume that the mapping $T: \Omega\subset \mathcal{P} \rightarrow X$ be
 such that $(I-T)$ is Lipschitz invertible with constant $\gamma>0,$  $F: \overline{U}\rightarrow X$  is a $k$-set contraction with
$0\le k<\gamma^{-1}$,  and $F(\overline{U})\subset(I-T)(\Omega).$
If
$$
Fx\neq (I-T)(\lambda x)\; \mbox{ for all } x\in\partial U\bigcap
\Omega,\,\lambda\geq 1\,\mbox{ and } \, \lambda x\in\Omega,
$$
then the mapping $T+F$ has at least one fixed point in $U\bigcap \Omega$.
\end{proposition}

In order to apply Proposition \ref{prop2c} and prove our main result, we will make the following assumptions.

There exist positive constants $\epsilon$, $A$, $R$ and  $b_1$  such that
   \begin{equation}
\label{H3}
\left\{\begin{array}{ccll}
\epsilon, A\in (0, 1),\quad 4A<\epsilon,\quad R&\geq& r,\quad b_1>1,\\
    \epsilon r+4\left(r+ \sum_{j=1}^l r^{p_j}\right)A&\leq& (\epsilon-4A)R,\\
    4\left(r+2R+ \sum_{j=1}^l r^{p_j}\right)A&<& \frac{1}{b_1}.
    \end{array}
\right.
\end{equation}
Denote
\begin{equation*}
B_1= \max\{1, 2L, 2L^2, 2L^3, 2L^4\}.
\end{equation*}

There exist  a non-negative function $g\in \mathcal{C}([0, \infty)\times [0, L])$ and $m\in (0,1)$  such that
\begin{equation}
\label{H4}
\left\{\begin{array}{ccll}
B_1(1+t+t^2+t^3+t^4) \dint_0^t \dint_0^L g(t_1, x_1) dx_1dt_1&\leq& A,\\\\
B_1(1+t+t^2+t^3+t^4)\dint_0^t \dint_0^L g(t_1, x_1) \dint_0^L \dint_0^{t_1}c_j(t_2, x_2) dt_2 dx_2 dx_1dt_1&\leq& A,\quad j\in \{1, \ldots, l\},\\\\
\frac{1-m}{4}\dint_1^{3\over 2}\dint_{\frac{L}{2}}^{\frac{2}{3}L}(2-t_1)^2(L-x_1)^2g(t_1, x_1)\dint_{\frac{L}{3}}^{\frac{L}{2}}x_2 u_0(x_2) dx_2 dx_1dt_1&\geq& \frac{A}{b_1}.
\end{array}
\right.
\end{equation}
In the last section we will give an example for  constants $\epsilon$, $A$, $r$, $R$, $m$, $b_1$ and for a function $g$ that satisfy \eqref{H3} and \eqref{H4}.

Let $E= \mathcal{C}^2([0, \infty)\times [0, L])$ be endowed with the norm
$$
\|u\|=\|u\|_\infty+\left\|\frac{\partial u}{\partial t}\right\|_\infty+\left\|\frac{\partial^2 u}{\partial t^2}\right\|_\infty+\left\|\frac{\partial u}{\partial x}\right\|_\infty+\left\|\frac{\partial^2 u}{\partial x^2}\right\|_\infty,
$$
provided it exists,
where $\|v\|_\infty=\displaystyle\sup_{(t, x)\in [0, \infty)\times [0, L]}|v(t,x)|.$
\begin{lemma}
\label{lemma1} Let $u\in E$ be a solution to the integral equation
\begin{equation}
\label{2}
\begin{array}{lll}
0&=& -\frac{1}{4} \int_0^t (t-t_1)^2 \int_0^x (x-x_1)^2 g(t_1, x_1) \int_0^{t_1} (t_1-t_2) u(t_2, x_1) dt_2 dx_1 dt_1\\
&&+\frac{1}{4}\int_0^t (t-t_1)^2 \int_0^x (x-x_1)^2g(t_1, x_1) \int_0^{x_1}x_2\bigg(-u(t_1, x_2)+u_0(x_2)+t_1 u_1(x_2)\\
&&+ \int_0^{t_1} (t_1-t_2)f(t_2, x_2, u(t_2, x_2))dt_2\bigg)dx_2 dx_1 dt_1\\
&&+ \frac{1}{4}\int_0^t (t-t_1)^2 \int_0^x x_1(x-x_1)^2 g(t_1, x_1)\int_{x_1}^L \bigg(-u(t_1, x_2)+u_0(x_2)+t_1 u_1(x_2)\\
&&+ \int_0^{t_1} (t_1-t_2)f(t_2, x_1, u(t_2, x_1)) dt_2\bigg) dx_2 dx_1 dt_1,\;\;\;(t, x)\in [0, \infty)\times [0, L].
\end{array}
\end{equation}

Then $u$ solves the IBVP \eqref{1}.
\end{lemma}
\begin{proof}
We differentiate trice in $t$ and then trice  in $x$ the equation \eqref{2} and we get
\begin{eqnarray*}
0&=& -g(t, x)\int_0^t (t-t_1) u(t_1, x) dt_1\\
&&+g(t, x) \int_0^x x_1\bigg( -u(t, x_1)+u_0(x_1) +t u_1(x_1)\\
&&+\int_0^t (t-t_1)f(t_1, x_1, u(t_1, x_1))dt_1\bigg)dx_1\\
&&+ x g(t, x) \int_x^L \bigg( -u(t, x_1) +u_0(x_1) +t u_1(x_1)\\
&&+ \int_0^t (t-t_1)f(t_1, x_1, u(t_1, x_1))dt_1\bigg)dx_1,
\end{eqnarray*}
$(t, x)\in [0, \infty)\times [0, L]$, whereupon
\begin{equation}
\label{3}
\begin{array}{lll}
0&=& -\int_0^t (t-t_1) u(t_1, x) dt_1\\
&&+ \int_0^x x_1\bigg( -u(t, x_1)+u_0(x_1) +t u_1(x_1)\\
&&+\int_0^t (t-t_1)f(t_1, x_1, u(t_1, x_1))dt_1\bigg)dx_1\\
&&+ x  \int_x^L \bigg( -u(t, x_1) +u_0(x_1) +t u_1(x_1)\\
&&+ \int_0^t (t-t_1)f(t_1, x_1, u(t_1, x_1))dt_1\bigg)dx_1,
\end{array}
\end{equation}
$(t, x)\in [0, \infty)\times [0, L]$. Now we differentiate the last equation with respect to $t$ and we find
\begin{equation}
\label{4}
\begin{array}{lll}
0&=& -\int_0^t u(t_1, x) dt_1\\
&&+ \int_0^x x_1\bigg( -u_t(t, x_1) +u_1(x_1) +\int_0^t f(t_1, x_1, u(t_1, x_1))dt_1\bigg) dx_1\\
&&+ x \int_x^L \bigg(-u_t(t, x_1) +u_1(x_1) +\int_0^tf(t_1, x_1, u(t_1, x_1))dt_1\bigg) dx_1,
\end{array}
\end{equation}
$(t, x)\in [0, \infty)\times [0, L]$, which we differentiate in $t$ and we arrive at
\begin{equation}
\label{5}
\begin{array}{lll}
0&=& -u(t, x)\\
&&+\int_0^x x_1\bigg(-u_{tt}(t, x_1)+f(t, x_1, u(t, x_1))\bigg) dx_1\\
&&+ x\int_x^L \bigg(-u_{tt}(t, x_1) +f(t, x_1, u(t, x_1))\bigg)dx_1,
\end{array}
\end{equation}
$(t, x)\in [0, \infty)\times [0, L]$. Now we differentiate with respect to $x$ the last equation and we find
\begin{equation}
\label{6}
\begin{array}{lll}
0&=& -u_x(t, x)\\
&&+ x\left(-u_{tt}(t, x) +f(t, x, u(t, x))\right)\\
&&-x\left(-u_{tt}(t, x)+f(t, x, u(t, x))\right)\\
&&+ \int_x^L \left(-u_{tt}(t, x_1)+f(t, x_1, u(t, x_1))\right)dx_1\\
&=& -u_x(t, x)\\
&&+ \int_x^L \left(-u_{tt}(t, x_1) +f(t, x_1, u(t, x_1))\right)dx_1,
\end{array}
\end{equation}
$(t, x)\in [0, \infty)\times [0, L]$. Now we differentiate the last equation with respect to $x$ and we find
\begin{eqnarray*}
0&=& -u_{xx}(t, x)+u_{tt}(t, x)-f(t, x, u(t, x)),\quad (t, x)\in [0, \infty)\times [0, L].
\end{eqnarray*}
We put $t=0$ in \eqref{3} and we find
\begin{eqnarray*}
0&=& \int_0^x x_1\left(-u(0, x_1)+u_0(x_1)\right)dx_1\\
&&+ x\int_x^L \left(-u(0, x_1)+u_0(x_1)\right)dx_1,\quad x\in [0, L],
\end{eqnarray*}
which we differentiate in $x$ and we get
\begin{eqnarray*}
0&=& x(-u(0, x)+u_0(x))+\int_x^L (-u(0, x_1) +u_0(x_1))dx_1\\
&&-x(-u(0, x)+u_0(x))\\
&=& \int_x^L (-u(0, x_1)+u_0(x_1))dx_1,\quad x\in [0, L],
\end{eqnarray*}
again we differentiate in $x$ and we find
\begin{equation*}
u(0, x)=u_0(x),\quad x\in [0, L].
\end{equation*}
Now we put $t=0$ in \eqref{4} and we get
\begin{eqnarray*}
0&=& \int_0^x x_1\left(-u_t(0, x_1)+u_1(x_1)\right)dx_1\\
&&+ x\int_x^L \left(-u_t(0, x_1)+u_1(x_1)\right)dx_1,\quad x\in [0, L],
\end{eqnarray*}
which we differentiate twice in $x$ and we find
\begin{equation*}
u_t(0, x)=u_1(x),\quad x\in [0, L].
\end{equation*}
Now we put $x=0$ in \eqref{5} and we get
\begin{equation*}
u(t, 0)=0,\quad t\in [0, \infty).
\end{equation*}
We put $x=L$ in \eqref{6} and we find
\begin{equation*}
u_x(t, L)=0,\quad t\in [0, \infty).
\end{equation*}
This completes the proof.
\end{proof}
For $u\in E$ and $(t, x)\in [0, \infty)\times [0, L]$, define
\begin{eqnarray*}
Gu(t, x)&=& -\frac{1}{4} \int_0^t (t-t_1)^2 \int_0^x (x-x_1)^2 g(t_1, x_1) \int_0^{t_1} (t_1-t_2) u(t_2, x_1) dt_2 dx_1 dt_1,\\
F_1u(t, x)&=& \int_0^x x_1\bigg(-u(t, x_1)+u_0(x_1) +t u_1(x_1)\\
&&+\int_0^t (t-t_1)f(t_1, x_1, u(t_1, x_1))dt_1\bigg)dx_1,\\
F_2 u(t, x)&=& \int_x^L \bigg(-u(t, x_1)+u_0(x_1)+t u_1(x_1)\\
&&+ \int_0^t (t-t_1) f(t_1, x_1, u(t_1, x_1))dt_1\bigg)dx_1,\\
F_3 u(t, x)&=& F_1 u(t, x)+x F_2 u(t, x),\\
Fu(t, x)&=&\frac{1}{4} \int_0^t \int_0^x (t-t_1)^2(x-x_1)^2 g(t_1, x_1)F_3 u(t_1, x_1)dx_1 dt_1.
\end{eqnarray*}
Observe that the equation \eqref{2} can be written in the form
\begin{equation*}
Gu(t, x)+Fu(t, x)=0,\quad (t, x)\in [0, \infty)\times [0, L],\quad u\in E.
\end{equation*}
\begin{lemma}
\label{lemma3.1.3.1} Suppose \eqref{H1} and \eqref{H2} are fulfilled. Then, for $u\in \mathcal{C}([0, \infty)\times [0, L])$, $|u|\leq r$ on $[0, \infty)\times [0, L]$, we have
\begin{eqnarray*}
|F_1u(t, x)| &\leq& 2L^2 r(1+t)+Lt \sum_{j=1}^l r^{p_j}\int_0^L \int_0^t c_j(t_1, x_1) dt_1 dx_1,\\
|F_2u(t, x)|&\leq& 2rL (1+t)+t \sum_{j=1}^l r^{p_j} \int_0^L \int_0^t c_j(t_1, x_1)dt_1 dx_1,
\end{eqnarray*}
and
\begin{equation}
\label{3.2.3.2} |F_3 u(t, x)|\leq 4L^2 r(1+t)+2Lt \sum_{j=1}^l r^{p_j}\int_0^L \int_0^t c_j(t_1, x_1) dt_1 dx_1.
\end{equation}
\end{lemma}
\begin{proof}
Let $u\in \mathcal{C}([0, \infty)\times [0, L])$ and $|u|\leq r$ on $[0, \infty)\times [0, L]$. Then, using \eqref{H1} and \eqref{H2}, we get
\begin{eqnarray*}
|F_1u(t, x)|&=& \bigg|\int_0^x x_1\bigg(-u(t, x_1)+u_0(x_1) +t u_1(x_1)\\
&&+\int_0^t (t-t_1)f(t_1, x_1, u(t_1, x_1))dt_1\bigg)dx_1\bigg|\\
&\leq& \int_0^x x_1\bigg( |u(t, x_1)|+|u_0(x_1)|+t |u_1(x_1)|\\
&&+ \int_0^t (t-t_1) |f(t_1, x_1, u(t_1, x_1))| dt_1\bigg) dx_1\\
&\leq& L\int_0^L\bigg((2+t)r+t \sum_{j=1}^l \int_0^t c_j(t_1, x_1)  |u(t_1, x_1)|^{p_j}dt_1\bigg)dx_1\\
&\leq& L^2 (2+t) r +Lt \sum_{j=1}^l r^{p_j} \int_0^L \int_0^t c_j(t_1, x_1) dt_1 dx_1\\
&\leq& 2L^2 r(1+t) +Lt \sum_{j=1}^l r^{p_j} \int_0^L \int_0^t c_j(t_1, x_1) dt_1 dx_1
\end{eqnarray*}
and
\begin{eqnarray*}
|F_2u(t, x)|&=& \bigg|\int_x^L \bigg(-u(t, x_1)+u_0(x_1) +t u_1(x_1)\\
&&+\int_0^t (t-t_1)f(t_1, x_1, u(t_1, x_1))dt_1\bigg)dx_1\bigg|
\\
&\leq& \int_x^L \bigg( |u(t, x_1)|+|u_0(x_1)|+t |u_1(x_1)|\\
&&+ \int_0^t (t-t_1) |f(t_1, x_1, u(t_1, x_1))| dt_1\bigg) dx_1
\end{eqnarray*}
\begin{eqnarray*}
&\leq& \int_0^L\bigg((2+t)r+t \sum_{j=1}^l \int_0^t c_j(t_1, x_1)  |u(t_1, x_1)|^{p_j}dt_1\bigg)dx_1\\
&\leq& L (2+t) r +t \sum_{j=1}^l r^{p_j} \int_0^L \int_0^t c_j(t_1, x_1) dt_1 dx_1\\
&\leq& 2L r(1+t) +t \sum_{j=1}^l r^{p_j} \int_0^L \int_0^t c_j(t_1, x_1) dt_1 dx_1,
\end{eqnarray*}
$(t, x)\in [0, \infty)\times [0, L]$. Hence,  we get \eqref{3.2.3.2}. This completes the proof.
\end{proof}
\begin{lemma}
\label{lemma2}
Suppose \eqref{H1}, \eqref{H2}, \eqref{H3} and \eqref{H4} are fulfilled. Then, for $u\in E$ and $\Vert u\Vert \leq r$, we have
\begin{eqnarray*}
\Vert Gu\Vert&\leq& 4rA,\\
\Vert Fu\Vert&\leq& 4\left(r+ \sum_{j=1}^l r^{p_j}\right)A.
\end{eqnarray*}
\end{lemma}
\begin{proof}
Using \eqref{H3} and the first inequality of \eqref{H4}, we have the following estimates
\begin{eqnarray*}
|Gu(t, x)|&=& \bigg|-\frac{1}{4}\int_0^t \int_0^x (t-t_1)^2 (x-x_1)^2 g(t_1, x_1) \int_0^{t_1} (t_1-t_2) u(t_2, x_1) dt_2 dx_1 dt_1\bigg|\\
&\leq& \frac{1}{4}\int_0^t \int_0^x (t-t_1)^2 (x-x_1)^2 g(t_1, x_1) \int_0^{t_1} (t_1-t_2) |u(t_2, x_1)| dt_2 dx_1 dt_1\\
&\leq& \frac{r}{4} \int_0^t \int_0^x t_1^2(t-t_1)^2 (x-x_1)^2 g(t_1, x_1) dx_1 dt_1\\
&\leq& \frac{r}{4} t^4L^2\int_0^t \int_0^L g(t_1, x_1) dx_1 dt_1\\
&\leq& r B_1 t^4 \int_0^t \int_0^L g(t_1, x_1) dx_1 dt_1\\
&\leq& rA,
\end{eqnarray*}
and
\begin{eqnarray*}
\left|\frac{\partial}{\partial t}Gu(t, x)\right|&=& \bigg|-\frac{1}{2}\int_0^t \int_0^x (t-t_1) (x-x_1)^2 g(t_1, x_1) \int_0^{t_1} (t_1-t_2) u(t_2, x_1) dt_2 dx_1 dt_1\bigg|
\\
&\leq& \frac{1}{2}\int_0^t \int_0^x (t-t_1) (x-x_1)^2 g(t_1, x_1) \int_0^{t_1} (t_1-t_2) |u(t_2, x_1)| dt_2 dx_1 dt_1\\
&\leq& \frac{r}{2}\int_0^t \int_0^x t_1^2 (t-t_1) (x-x_1)^2 g(t_1, x_1) dx_1 dt_1\\
&\leq& \frac{r}{2}t^3 L^2\int_0^t \int_0^L g(t_1, x_1) dx_1 dt_1\\
&\leq& rB_1 t^3 \int_0^t \int_0^L g(t_1, x_1) dx_1 dt_1\\
&\leq& rA,
\end{eqnarray*}
and
\begin{eqnarray*}
\left|\frac{\partial^2}{\partial t^2}Gu(t, x)\right|&=&\bigg|- \frac{1}{2}\int_0^t \int_0^x  (x-x_1)^2 g(t_1, x_1) \int_0^{t_1} (t_1-t_2) u(t_2, x_1) dt_2 dx_1 dt_1\bigg|
\\
&\leq& \frac{1}{2} \int_0^t \int_0^x (x-x_1)^2 g(t_1, x_1) \int_0^{t_1} (t_1-t_2) |u(t_2, x_1)| dt_2 dx_1 dt_1\\
&\leq& \frac{r}{2} \int_0^t \int_0^x t_1^2 (x-x_1)^2 g(t_1, x_1) dx_1 dt_1
\\
&\leq& \frac{r}{2} t^2 L^2 \int_0^t \int_0^L g(t_1, x_1) dx_1 dt_1\\
&\leq& rB_1 t^2 \int_0^t \int_0^L g(t_1, x_1) dx_1 dt_1\\
&\leq& rA,
\end{eqnarray*}
and
\begin{eqnarray*}
\left|\frac{\partial}{\partial x}Gu(t, x)\right|&=& \bigg|-\frac{1}{2}\int_0^t \int_0^x (t-t_1)^2 (x-x_1) g(t_1, x_1) \int_0^{t_1} (t_1-t_2) u(t_2, x_1) dt_2 dx_1 dt_1\bigg|
\\
&\leq& \frac{1}{2}\int_0^t \int_0^x (t-t_1)^2 (x-x_1) g(t_1, x_1) \int_0^{t_1} (t_1-t_2) |u(t_2, x_1)| dt_2 dx_1 dt_1\\
&\leq& \frac{r}{2}\int_0^t \int_0^x t_1^2 (t-t_1)^2 (x-x_1) g(t_1, x_1) dx_1 dt_1\\
&\leq& \frac{r}{2}t^4 L\int_0^t \int_0^L g(t_1, x_1) dx_1 dt_1\\
&\leq& rB_1 t^4 \int_0^t \int_0^L g(t_1, x_1) dx_1 dt_1\\
&\leq& rA,
\end{eqnarray*}
and
\begin{eqnarray*}
\left|\frac{\partial^2}{\partial x^2}Gu(t, x)\right|&=&\bigg|- \frac{1}{2}\int_0^t \int_0^x (t-t_1)^2  g(t_1, x_1) \int_0^{t_1} (t_1-t_2) u(t_2, x_1) dt_2 dx_1 dt_1\bigg|\\
&\leq& \frac{1}{2} \int_0^t \int_0^x (t-t_1)^2 g(t_1, x_1) \int_0^{t_1} (t_1-t_2) |u(t_2, x_1)| dt_2 dx_1 dt_1\\
&\leq& \frac{r}{2} \int_0^t \int_0^x t_1^2 (t-t_1)^2 g(t_1, x_1) dx_1 dt_1\\
&\leq& \frac{r}{2} t^4  \int_0^t \int_0^L g(t_1, x_1) dx_1 dt_1\\
&\leq& rB_1 t^4 \int_0^t \int_0^L g(t_1, x_1) dx_1 dt_1\\
&\leq& rA,\quad (t, x)\in [0, \infty)\times [0, L].
\end{eqnarray*}
Thus,
\begin{equation*}
\Vert Gu\Vert \leq 4rA.
\end{equation*}
Next, using \eqref{H3} and the first and second inequalities of \eqref{H4}, we get
\begin{eqnarray*}
\lefteqn{|Fu(t, x)| = \bigg| \frac{1}{4} \int_0^t \int_0^x (t-t_1)^2 (x-x_1)^2 g(t_1, x_1) F_3 u(t_1, x_1)dx_1 dt_1\bigg|}\\
&\leq& \frac{1}{4}\int_0^t \int_0^x (t-t_1)^2 (x-x_1)^2 g(t_1, x_1) |F_3u(t_1, x_1)| dx_1 dt_1
\\
&\leq& L^2 r \int_0^t \int_0^x (1+t_1)(t-t_1)^2 (x-x_1)^2 g(t_1, x_1) dx_1 dt_1\\
&&+\frac{L}{2}\sum_{j=1}^l r^{p_j} \int_0^t \int_0^x t_1(t-t_1)^2 (x-x_1)^2 g(t_1, x_1)\int_0^L \int_0^{t_1} c_j(t_2, x_2) dt_2 dx_2 dx_1 dt_1\\
&\leq& L^4 r (t^2+t^3)\int_0^t \int_0^L g(t_1, x_1) dx_1 dt_1\\
&&+ \frac{L^3}{2} \sum_{j=1}^l r^{p_j}t^3 \int_0^t \int_0^L g(t_1, x_1) \int_0^L \int_0^{t_1} c_j(t_2, x_2) dt_2 dx_2 dx_1 dt_1\\
&\leq& B_1 r (t^2+t^3)\int_0^t \int_0^L g(t_1, x_1) dx_1 dt_1\\
&&+B_1 \sum_{j=1}^l r^{p_j} t^3 \int_0^t \int_0^L g(t_1, x_1) '\int_0^L \int_0^{t_1}c_j(t_2, x_2) dt_2 dx_2 dx_1 dt_1\\
&\leq& \left(r+\sum_{j=1}^l r^{p_j}\right)A,\quad (t, x)\in [0, \infty)\times [0, L],
\end{eqnarray*}
and
\begin{eqnarray*}
\lefteqn{\bigg|\frac{\partial}{\partial t}Fu(t, x)\bigg| = \bigg| \frac{1}{2} \int_0^t \int_0^x (t-t_1) (x-x_1)^2 g(t_1, x_1) F_3 u(t_1, x_1)dx_1 dt_1\bigg|}\\
&\leq& \frac{1}{2}\int_0^t \int_0^x (t-t_1) (x-x_1)^2 g(t_1, x_1) |F_3u(t_1, x_1)| dx_1 dt_1
\\
&\leq& 2L^2 r \int_0^t \int_0^x (1+t_1)(t-t_1) (x-x_1)^2 g(t_1, x_1) dx_1 dt_1\\
&&+{L}\sum_{j=1}^l r^{p_j} \int_0^t \int_0^x t_1(t-t_1) (x-x_1)^2 g(t_1, x_1)\int_0^L \int_0^{t_1} c_j(t_2, x_2) dt_2 dx_2 dx_1 dt_1
\\
&\leq& 2L^4 r (t+t^2)\int_0^t \int_0^L g(t_1, x_1) dx_1 dt_1\\
&&+ \frac{L^3}{2} \sum_{j=1}^l r^{p_j}t^2 \int_0^t \int_0^L g(t_1, x_1) \int_0^L \int_0^{t_1} c_j(t_2, x_2) dt_2 dx_2 dx_1 dt_1
\\
&\leq& B_1 r (t+t^2)\int_0^t \int_0^L g(t_1, x_1) dx_1 dt_1
\end{eqnarray*}
\begin{eqnarray*}
&&+B_1 \sum_{j=1}^l r^{p_j} t^2 \int_0^t \int_0^L g(t_1, x_1) '\int_0^L \int_0^{t_1}c_j(t_2, x_2) dt_2 dx_2 dx_1 dt_1\\
&\leq& \left(r+\sum_{j=1}^l r^{p_j}\right)A,\quad (t, x)\in [0, \infty)\times [0, L],
\end{eqnarray*}
and
\begin{eqnarray*}
\lefteqn{\bigg|\frac{\partial^2}{\partial t^2}Fu(t, x)\bigg| = \bigg| \frac{1}{2} \int_0^t \int_0^x  (x-x_1)^2 g(t_1, x_1) F_3 u(t_1, x_1)dx_1 dt_1\bigg|}\\
&\leq& \frac{1}{2}\int_0^t \int_0^x  (x-x_1)^2 g(t_1, x_1) |F_3u(t_1, x_1)| dx_1 dt_1
\\
&\leq& 2L^2 r \int_0^t \int_0^x (1+t_1) (x-x_1)^2 g(t_1, x_1) dx_1 dt_1\\
&&+{L}\sum_{j=1}^l r^{p_j} \int_0^t \int_0^x t_1 (x-x_1)^2 g(t_1, x_1)\int_0^L \int_0^{t_1} c_j(t_2, x_2) dt_2 dx_2 dx_1 dt_1
\\
&\leq& 2L^4 r (1+t)\int_0^t \int_0^L g(t_1, x_1) dx_1 dt_1\\
&&+ \frac{L^3}{2} \sum_{j=1}^l r^{p_j}t \int_0^t \int_0^L g(t_1, x_1) \int_0^L \int_0^{t_1} c_j(t_2, x_2) dt_2 dx_2 dx_1 dt_1
\\
&\leq& B_1 r (1+t)\int_0^t \int_0^L g(t_1, x_1) dx_1 dt_1\\
&&+B_1 \sum_{j=1}^l r^{p_j} t \int_0^t \int_0^L g(t_1, x_1) '\int_0^L \int_0^{t_1}c_j(t_2, x_2) dt_2 dx_2 dx_1 dt_1\\
&\leq& \left(r+\sum_{j=1}^l r^{p_j}\right)A,\quad (t, x)\in [0, \infty)\times [0, L],
\end{eqnarray*}
and
\begin{eqnarray*}
\lefteqn{\bigg|\frac{\partial}{\partial x}Fu(t, x)\bigg| = \bigg| \frac{1}{2} \int_0^t \int_0^x (t-t_1)^2 (x-x_1) g(t_1, x_1) F_3 u(t_1, x_1)dx_1 dt_1\bigg|}
\\
&\leq& \frac{1}{2}\int_0^t \int_0^x (t-t_1)^2 (x-x_1) g(t_1, x_1) |F_3u(t_1, x_1)| dx_1 dt_1
\\
&\leq& 2L^2 r \int_0^t \int_0^x (1+t_1)(t-t_1)^2 (x-x_1) g(t_1, x_1) dx_1 dt_1\\
&&+{L}\sum_{j=1}^l r^{p_j} \int_0^t \int_0^x t_1(t-t_1)^2 (x-x_1) g(t_1, x_1)\int_0^L \int_0^{t_1} c_j(t_2, x_2) dt_2 dx_2 dx_1 dt_1
\end{eqnarray*}
\begin{eqnarray*}
&\leq& 2 L^3 r (t^2+t^3)\int_0^t \int_0^L g(t_1, x_1) dx_1 dt_1\\
&&+ {L^2} \sum_{j=1}^l r^{p_j}t^3 \int_0^t \int_0^L g(t_1, x_1) \int_0^L \int_0^{t_1} c_j(t_2, x_2) dt_2 dx_2 dx_1 dt_1\\
&\leq& B_1 r (t^2+t^3)\int_0^t \int_0^L g(t_1, x_1) dx_1 dt_1\\
&&+B_1 \sum_{j=1}^l r^{p_j} t^3 \int_0^t \int_0^L g(t_1, x_1) '\int_0^L \int_0^{t_1}c_j(t_2, x_2) dt_2 dx_2 dx_1 dt_1\\
&\leq& \left(r+\sum_{j=1}^l r^{p_j}\right)A,\quad (t, x)\in [0, \infty)\times [0, L],
\end{eqnarray*}
and
\begin{eqnarray*}
\lefteqn{\bigg|\frac{\partial^2}{\partial x^2}Fu(t, x)\bigg| = \bigg| \frac{1}{2} \int_0^t \int_0^x (t-t_1)^2  g(t_1, x_1) F_3 u(t_1, x_1)dx_1 dt_1\bigg|}
\\
&\leq& \frac{1}{2}\int_0^t \int_0^x (t-t_1)^2  g(t_1, x_1) |F_3u(t_1, x_1)| dx_1 dt_1
\\
&\leq& 2L^2 r \int_0^t \int_0^x (1+t_1)(t-t_1)^2  g(t_1, x_1) dx_1 dt_1\\
&&+{L}\sum_{j=1}^l r^{p_j} \int_0^t \int_0^x t_1(t-t_1)^2  g(t_1, x_1)\int_0^L \int_0^{t_1} c_j(t_2, x_2) dt_2 dx_2 dx_1 dt_1
\\
&\leq& 2 L^2 r (t^2+t^3)\int_0^t \int_0^L g(t_1, x_1) dx_1 dt_1\\
&&+ {L} \sum_{j=1}^l r^{p_j}t^3 \int_0^t \int_0^L g(t_1, x_1) \int_0^L \int_0^{t_1} c_j(t_2, x_2) dt_2 dx_2 dx_1 dt_1
\end{eqnarray*}
\begin{eqnarray*}
&\leq& B_1 r (t^2+t^3)\int_0^t \int_0^L g(t_1, x_1) dx_1 dt_1\\
&&+B_1 \sum_{j=1}^l r^{p_j} t^3 \int_0^t \int_0^L g(t_1, x_1) '\int_0^L \int_0^{t_1}c_j(t_2, x_2) dt_2 dx_2 dx_1 dt_1\\
&\leq& \left(r+\sum_{j=1}^l r^{p_j}\right)A,\quad (t, x)\in [0, \infty)\times [0, L].
\end{eqnarray*}
Therefore
\begin{equation*}
\Vert Fu\Vert \leq 4\left(r+ \sum_{j=1}^l r^{p_j}\right)A.
\end{equation*}
This completes the proof.
\end{proof}
\section{Proof of the Main Result}
For $u\in E$, define the mappings
\begin{eqnarray*}
T u(t, x)&=& (1-\epsilon) u(t, x) +Gu(t, x),\\
S u(t, x)&=& \epsilon u(t, x) +F u(t, x),\quad (t, x)\in [0, \infty)\times [0, L].
\end{eqnarray*}
Note that if $u\in E$ is a fixed point of the mapping $T+S$, then
\begin{eqnarray*}
u(t, x)&=& T u(t, x)+Su(t, x)\\
&=& (1-\epsilon)u(t, x)+Gu(t, x)+\epsilon u(t, x)+Fu(t, x)\\
&=& u(t, x)+Fu(t, x)+Gu(t, x),\quad (t, x)\in [0, \infty)\times [0, L],
\end{eqnarray*}
or
\begin{eqnarray*}
0=Gu(t, x)+Fu(t, x),\quad (t, x)\in [0, \infty)\times [0, L].
\end{eqnarray*}
Therefore any fixed point $u\in E$ of the mapping $T+S$ is a solution of the IBVP \eqref{1}.
Define
\begin{eqnarray*}
\widetilde{\mathcal{P}}&=& \{u\in E: u(t, x)\geq 0,\quad (t, x)\in [0, \infty)\times [0, L]\},
\end{eqnarray*}
Let $\mathcal{P}$ be the set of all equi-continuous families in $\mathcal{P}$( an example for an equi-continuous family in $\mathcal{P}$ is the family $\{(3+\sin(t+n))(3+\cos(x+n)),\quad t\in [0, \infty),\quad x\in [0, L]\}_{n\in \mathbb{N}}$). Let also,
\begin{eqnarray*}
\Omega&=& \{u\in {\mathcal{P}}:\Vert u\Vert \leq R\},\\
U&=& \bigg\{ u\in \mathcal{P}:  u(t, x)< u_0(x), \quad (t, x)\in (0, \infty)\times [0, L],\\
&& u(t, x)< mu_0(x),\quad (t, x)\in [1, 2]\times [0, L],\quad \Vert u\Vert <r\bigg\}.
\end{eqnarray*}
Note that for $u\in \overline{U}$ , we have $F_1u\geq 0$, $F_2u\geq 0$, $Fu\geq 0$. Hence, for $u\in \overline{U}$, we have
\begin{eqnarray*}
F_1u(t, x)&\geq& \int_0^x x_1 (-u(t, x_1)+u_0(x_1))dx_1,\quad (t, x)\in [0, \infty)\times [0, L],
\end{eqnarray*}
and
\begin{equation}
\label{**}
Fu(t, x)\geq \frac{1}{4}\int_0^t \int_0^x (t-t_1)^2 (x-x_1)^2 g(t_1, x_1)\int_0^{x_1} x_2(-u(t_1, x_2)+u_0(x_2))dx_2 dx_1 dt_1,
\end{equation}
$(t, x)\in [0, \infty)\times [0, L]$.
\begin{enumerate}
\item For $u\in \Omega$, we have
\begin{equation*}
(I-T) u(t, x)= \epsilon u(t, x)-G u(t, x),\quad (t, x)\in [0, \infty)\times [0, L].
\end{equation*}
Then, for $u\in \Omega$, using Lemma \ref{lemma2}, we find
\begin{eqnarray*}
\Vert (I-T)u\Vert&\leq& \epsilon \Vert u\Vert +\Vert Gu\Vert\\
&\leq& (\epsilon+4A) \Vert u\Vert,\\
\Vert (I-T) u\Vert&\geq& \epsilon \Vert u\Vert -\Vert Gu\Vert \\
&\geq& (\epsilon-4A)\Vert u\Vert.
\end{eqnarray*}
Thus, $I-T: \Omega\to E$ is Lipschitz invertible with a constant $\gamma\in \left[\frac{1}{\epsilon+4A}, \frac{1}{\epsilon-4A}\right]$.
\item Let $u\in \overline{U}$. By Lemma \ref{lemma2}, we have
\begin{eqnarray*}
\Vert Su\Vert &\leq& \epsilon \Vert u\Vert +\Vert Fu\Vert\\
&\leq& \epsilon r+4\left(r+ \sum_{j=1}^l r^{p_j}\right)A.
\end{eqnarray*}
Therefore $S: \overline{U}\to E$ is uniformly bounded. Since $S: \overline{U}\to \mathbb{R}$  is continuous, we have that $S(\overline{U})$ is equi-continuous and $S: \overline{U}\to E$ is  relatively compact. Therefore $S: \overline{U}\to E$ is a $0$-set contraction.
\item Let $u\in \overline{U}$. For $z\in \Omega$, define the mapping
\begin{equation*}
L z(t, x)= T z(t, s) +Su(t, s),\quad (t, x)\in [0, \infty)\times [0, L].
\end{equation*}
For $z\in \Omega$, we get
\begin{eqnarray*}
\Vert L z\Vert&=& \Vert Tz+Su\Vert\\
&\leq& \Vert Tz\Vert +\Vert Su\Vert \\
&\leq& (1-\epsilon +4A)R+\epsilon r +4\left(r+ \sum_{j=1}^l r^{p_j}\right)A\\
&\leq& (1-\epsilon+4A)R+(\epsilon-4A)R\\
&=& R,
\end{eqnarray*}
i.e., $L:\Omega\to \Omega$. Next, for $z_1, z_2\in \Omega$, we have
\begin{eqnarray*}
\Vert Lz_1-L z_2\Vert &=& \Vert Tz_1-Tz_2\Vert\\
&=& \Vert T(z_1-z_2)\Vert\\
&\leq& (1-\epsilon+4A)\Vert z_1-z_2\Vert.
\end{eqnarray*}
Therefore $L:\Omega\to \Omega$ is a contraction mapping. Hence, there exists a unique $z\in \Omega$ such that
\begin{equation*}
z=Lz
\end{equation*}
or
\begin{equation*}
(I-T)z=Su.
\end{equation*}
Consequently $S(\overline{U})\subset (I-T)(\Omega)$.
\item Assume that there are $u\in \partial U$ and $\lambda\geq 1$ such that
\begin{equation*}
Su=(I-T)(\lambda u),\quad \lambda u\in  \Omega.
\end{equation*}
We have
\begin{equation*}
\epsilon u+Fu=\epsilon \lambda u-G(\lambda u)
\end{equation*}
or
\begin{equation*}
\epsilon(\lambda-1) u=Fu+G(\lambda u).
\end{equation*}
Since $\lambda u\in \Omega$, we have that $\Vert \lambda u\Vert \leq R$. Hence and Lemma \ref{lemma2}, we obtain $\Vert G(\lambda u)\Vert \leq 4RA$. Then
\begin{eqnarray*}
\epsilon (\lambda-1) r&=& \epsilon(\lambda-1)\Vert u\Vert \\
&=& \Vert Fu- G(\lambda u)\Vert \\
&\leq& \Vert Fu\Vert +\Vert G(\lambda u)\Vert \\
&\leq& 4\left(r+R+\sum_{j=1}^l r^{p_j}\right)A.
\end{eqnarray*}
Hence, for $\lambda u\in \Omega$  and $u\in \partial U$, using \eqref{**}, we get
\begin{eqnarray*}
\lefteqn{4\left( r+R+\sum_{j=1}^l r^{p_j}\right)A \geq \epsilon(\lambda-1) \Vert u\Vert}\\
&=& \Vert Fu+G(\lambda u)\Vert\\
&\geq& \Vert Fu\Vert -\Vert G(\lambda u)\Vert\\
&\geq& \sup_{(t, x)\in [0, \infty)\times [0, L]}Fu(t, x)-\Vert G(\lambda u)\Vert\\
&\geq& Fu(2, L)-\Vert G(\lambda u)\Vert\\
&\geq& \frac{1}{4} \int_0^2 \int_0^L (2-t_1)^2(L-x_1)^2g(t_1, x_1)\int_0^{x_1}x_2\left(-u(t_1, x_2)+u_0(x_2)\right)dx_2 dx_1 dt_1-4AR\\
&\geq& \frac{1}{4} \int_1^2 \int_0^L (2-t_1)^2(L-x_1)^2g(t_1, x_1)\int_0^{x_1}x_2\left(-u(t_1, x_2)+u_0(x_2)\right)dx_2 dx_1 dt_1-4AR\\
&\geq& \frac{1-m}{4} \int_1^2 \int_{L\over 2}^L (2-t_1)^2(L-x_1)^2g(t_1, x_1)\int_0^{x_1}x_2 u_0(x_2) dx_2 dx_1 dt_1-4AR\\
&\geq& \frac{1-m}{4}\int_1^2 \int_{L\over 2}^L(2-t_1)^2(L-x_1)^2g(t_1, x_1)\int_0^{L\over 2}x_2 u_0(x_2) dx_2 dx_1 dt_1-4AR\\
&\geq& \frac{1-m}{4}\int_1^2 \int_{L\over 2}^L(2-t_1)^2 (L-x_1)^2 g(t_1, x_1) \int_{L\over 3}^{L\over 2}x_2 u_0(x_2)dx_2 dx_1 dt_1-4AR\\
&\geq& \frac{1-m}{4} \int_1^{3\over 2} \int_{\frac{L}{2}}^{\frac{2}{3}L} (2-t_1)^2(L-x_1)^2g(t_1, x_1)\int_{\frac{L}{3}}^{\frac{L}{2}}x_2 u_0(x_2) dx_2 dx_1 dt_1 -4AR\\
&\geq& \frac{A}{b_1}-4AR,
\end{eqnarray*}
whereupon
\begin{equation*}
4\left(r+2R+\sum_{j=1}^l r^{p_j}\right)A\geq \frac{A}{b_1}
\end{equation*}
or
\begin{equation*}
4\left(r+2R+\sum_{j=1}^l r^{p_j}\right)\geq \frac{1}{b_1}.
\end{equation*}
This is a contradiction.
\end{enumerate}
By (1)--(4) and Proposition \ref{prop2c}, we conclude that the mapping $T+S$ has a fixed point in $U$. This completes the proof.
\section{Example}
Consider the following IBVP
\begin{equation}
\label{4.1}
\begin{array}{lll}
u_{tt}-u_{xx}&=& |u|^p,\quad t\geq 0,\quad x\in [0, 1],\\
u(0, x)&=& \frac{1}{10} x(1-x)^2,\quad x\in [0, 1],\\
u_t(0, x)&=& \frac{1}{50}x(1-x)^2,\quad x\in [0, 1],\\
u(t, 0)=u_x(t, 1)&=& 0,\quad t\geq 0,
\end{array}
\end{equation}
where $p>1$.
 Here
\begin{equation*}
f(t, x, u)= |u|^p,\quad L=l=1,\quad B_1=2,\quad c_1(t, x)=1,\quad (t, x)\in [0, \infty)\times [0, 1],
\end{equation*}
and
\begin{equation*}
u_0(x)=\frac{1}{10} x(1-x)^2,\quad u_1(x)=\frac{1}{50} x(1-x)^2,\quad x\in [0, 1].
\end{equation*}
Now we will construct a function $g$ so that \eqref{H4} holds. Let
\begin{equation*}
h(t)= \log \frac{1+t^4 \sqrt{2}+t^8}{1-t^4\sqrt{2}+t^8},\quad l(t)= \arctan\frac{t^4 \sqrt{2}}{1-t^8},\quad t\geq 0.
\end{equation*}
We have
\begin{eqnarray*}
h^{\prime}(t)&=& \frac{1}{(1+t^4 \sqrt{2}+t^8)(1-t^4\sqrt{2}+t^8)}\bigg((4\sqrt{2}t^3+8t^7)(1-t^4 \sqrt{2}+t^8)\\
&&-(1+t^4\sqrt{2}+t^8)(-4\sqrt{2}t^3+8t^7)\bigg)\\
&=&  \frac{1}{(1+t^4 \sqrt{2}+t^8)(1-t^4\sqrt{2}+t^8)}\bigg( 4\sqrt{2}t^3-8t^7+4\sqrt{2}t^{11}+8t^7\\
&&-8\sqrt{2}t^{11}+8t^{15}+4\sqrt{2}t^3-8t^7+8t^7-8\sqrt{2}t^{11}+4\sqrt{2}t^{11}-8t^{15}\bigg)\\
&=&  -\frac{8\sqrt{2}t^3(t^8-1)}{(1+t^4 \sqrt{2}+t^8)(1-t^4\sqrt{2}+t^8)},\quad t\geq 0.
\end{eqnarray*}
Thus,
\begin{equation*}
\sup_{t\geq 0} h(t)=h(1)=\log \frac{2+\sqrt[4]{2}}{2-\sqrt[4]{2}},
\end{equation*}
$h$ is an increasing function on $[0, 1]$ and it is a decreasing function on $[1, \infty)$. Next,
\begin{eqnarray*}
l^{\prime}(t)&=&\frac{1}{1+\frac{2t^8}{(1-t^8)^2}}\frac{4\sqrt{2}t^3(1-t^8)+8t^7t^4 \sqrt{2}}{(1-t^8)^2}\\
&=& \frac{4\sqrt{2}t^3-4\sqrt{2}t^{11}+8\sqrt{2}t^{11}}{1+t^{16}}\\
&=& \frac{4\sqrt{2}t^3(1+t^8)}{1+t^{16}},\quad t\geq 0.
\end{eqnarray*}
Therefore $l$ is an increasing function on $[0, \infty)$.   Note that, by l'Hopital's rule, we have
\begin{eqnarray*}
\lim_{t\to\infty}t h(t)&=& 0\\
\lim_{t\to\infty} t^2 h(t)&=&0\\
\lim_{t\to\infty} t^3 h(t)&=& 0\\
\lim_{t\to\infty} t^4 h(t)&=&2\sqrt{2},
\end{eqnarray*}
and
\begin{eqnarray*}
\lim_{t\to\infty} tl(t)&=&0\\
\lim_{t\to\infty}t^2 l(t)&=&0\\
\lim_{t\to\infty} t^3l(t)&=&0\\
\lim_{t\to\infty} t^4l(t)&=&-\sqrt{2}.
\end{eqnarray*}
Consequently, there exists a constant $B>1$ such that
\begin{equation*}
(1+t+t^2+t^3+t^4)\left(\frac{1}{16\sqrt{2}}\log\frac{1+t^4\sqrt{2}+t^8}{1-t^4\sqrt{2}+t^8}+\frac{1}{8\sqrt{2}}\arctan \frac{t^4\sqrt{2}}{1-t^8}\right)\leq B.
\end{equation*}
Note that, by \cite{pol}(pp. 707, Integral 79), we have
\begin{equation*}
\int\frac{dz}{1+z^4}=\frac{1}{4\sqrt{2}}\log\frac{1+z\sqrt{2}+z^2}{1-z\sqrt{2}+z^2}+\frac{1}{2\sqrt{2}}\arctan\frac{z\sqrt{2}}{1-z^2}.
\end{equation*}
Take
\begin{equation*}
 \epsilon=\frac{1}{2},\quad b_1=\frac{B(15)^2(2^{16}+3^{16})(2^2+3^2)3^4}{2^5},\quad A= \frac{1}{20 b_1},
\end{equation*}
\begin{equation*}
m=\frac{1}{2},\quad r=\frac{4}{27},
\end{equation*}
\begin{equation*}
R= \frac{2}{\epsilon-4A}\left(\epsilon r+4\left(r+r^p\right)A\right).
\end{equation*}
 Then
\begin{equation*}
0\leq u_0(x)<r,\quad 0\leq u_1(x)<r,\quad x\in [0, 1],
\end{equation*}
and
\begin{equation*}
u_0(0)=u_{0x}(1)=u_1(0)=u_{1x}(1)=0,
\end{equation*}
i.e., \eqref{H2} holds.  Also, \eqref{H3} holds.
Let
\begin{equation*}
g(t, x)= \frac{A}{200B}\frac{t^3}{(1+t^{16})(1+t^2)},\quad (t, x)\in [0, \infty)\times [0, 1].
\end{equation*}
\begin{eqnarray*}
\lefteqn{B_1(1+t+t^2+t^3+t^4) \int_0^t \int_0^1 g(t_1, x_1) dx_1 dt_1}\\
&\leq& \frac{A}{100B} (1+t+t^2+t^3+t^4)\int_0^t \frac{t_1^3}{1+t_1^{16}}dt_1\\
&\leq& \frac{A}{100}\\
&\leq& A
\end{eqnarray*}
and
\begin{eqnarray*}
\lefteqn{B_1(1+t+t^2+t^3+t^4) \int_0^t \int_0^1 g(t_1, x_1)\int_0^1 \int_0^{t_1} c_j(t_1, x_1)dt_2 dx_2  dx_1 dt_1}\\
&\leq& 2(1+t+t^2+t^3+t^4) \int_0^t \int_0^1 t_1g(t_1, x_1) dx_1 dt_1\\
&\leq& \frac{A}{100B} (1+t+t^2+t^3+t^4)\int_0^t \frac{t_1^3}{1+t_1^{16}}dt_1\\
&\leq& \frac{A}{100}\\
&\leq& A,
\end{eqnarray*}
and
\begin{eqnarray*}
\lefteqn{\frac{1-m}{4} \int_1^{3\over 2} \int_{1\over 2}^{2\over 3} (2-t_1)^2(1-x_1)^2g(t_1, x_1) \int_{1\over 3}^{1\over 2} x_2 u_0(x_2) dx_2 dx_1 dt_1}\\
&=& \frac{A}{1600B} \int_1^{3\over 2} \int_{1\over 2}^{2\over 3} (2-t_1)^2(1-x_1)^2 \frac{t_1^3}{(1+t_1^{16})(1+t_1^2)}\int_{1\over 3}^{1\over 2} x_2^2(1-x_2)^2 dx_2 dx_1 dt_1\\
&\geq& \frac{A}{1600B} \left(\frac{1}{2}\right)^2\left(\frac{1}{3}\right)^2\frac{1}{\left(1+\left(\frac{3}{2}\right)^{16}\right)\left(1+ \left(\frac{3}{2}\right)^2\right)}\left(\frac{1}{3}\right)^2\left(\frac{1}{2}\right)^2\frac{1}{2}\left(\frac{1}{6}\right)^2\\
&=& \frac{A}{1600B}\cdot \frac{1}{2^{7}\cdot 3^6}\frac{2^{18}}{(2^{16}+3^{16})(2^2+3^2)}\\
&\geq& \frac{2^5A}{B(15)^2(2^{16}+3^{16})(2^2+3^2)3^4}= \frac{A}{b_1}.
\end{eqnarray*}
Consequently  \eqref{H4} holds and the IBVP \eqref{4.1}  has at least one non-negative solution $u\in \mathcal{C}^2([0, \infty)\times [0, 1])$.

\end{document}